\documentclass[12pt]{amsart}

\usepackage{amsmath}
\usepackage{amssymb}
\usepackage{amsfonts}
\usepackage{amsthm}
\usepackage{bbm}
\usepackage{cite}
\usepackage{latexsym}
\usepackage[pdftex]{graphicx}
\usepackage{color}
\usepackage{comment}
\usepackage{url}
\usepackage{hyperref}
\usepackage{todonotes}
\usepackage{mathtools}
\usepackage{empheq}
\usepackage{multicol}
\usepackage{bm}
\usepackage[all]{xy}
\xyoption{all}
\usepackage{mdframed} 
\usepackage{stmaryrd}

\newcommand{\R}{\mathbb{R}}

\newcommand{\Z}{\mathbb{Z}}
\newcommand{\Q}{\mathbb{Q}}
\newcommand{\D}{\mathbb{D}}
\newcommand{\F}{\mathbb{F}}

\newcommand{\scr}{\mathcal}

\newcommand{\wt}{\widetilde}

\newcommand*\quot[2]{{^{\textstyle #1}\big/_{\textstyle #2}}}

\newcommand{\xtwoheadrightarrow}[2][]{%
  \xrightarrow[#1]{#2}\mathrel{\mkern-14mu}\rightarrow
}

\theoremstyle{plain}
\newtheorem{thm}{Theorem}[section]

\newtheorem{lem}[thm]{Lemma}
\newtheorem{prop}[thm]{Proposition}
\newtheorem{quest}[thm]{Question}
\theoremstyle{definition}
\newtheorem{defn}[thm]{Definition}
\newtheorem{convn}[thm]{Convention}

\newtheorem{exam}[thm]{Example}
\newtheorem{cor}[thm]{Corollary}

\newtheorem*{claim*}{Claim}
\theoremstyle{remark}
\newtheorem{rmk}[thm]{Remark}

\title[Exact submanifold obstructions]{Obstructions for exact submanifolds with symplectic applications}

\author{Kevin Sackel}
\address{Department of Mathematics and Statistics, University of Massachusetts, Amherst, MA 01003}
\email{k.sackel.math@gmail.com \\ ksackel@umass.edu}

\begin{document}
\begin{abstract}
	Suppose $X^{N}$ is a closed oriented manifold, $\alpha \in H^*(X;\R)$ is a cohomology class, and $Z \in H_{N-k}(X)$ is an integral homology class. We ask the following question: is there an oriented embedded submanifold $Y^{N-k} \subset X$ with homology class $Z$ such that $\alpha|_Y = 0 \in H^*(Y;\R)$? In this article, we provide a family of computable obstructions to the existence of such `exact' submanifolds in a given homology class which arise from studying formal deformations of the de Rham complex. In the final section, we apply these obstructions to prove that the following symplectic manifolds admit no non-separating exact (a fortiori contact-type) hypersurfaces: K\"ahler manifolds, symplectically uniruled manifolds, and the Kodaira--Thurston manifold.
\end{abstract}

\maketitle


\section{Introduction}

Suppose $X^n$ is a closed and oriented manifold of dimension $n$, and let $\alpha \in H^r(X;\R)$ be a fixed cohomology class of homogeneous degree $r$.\footnote{Having $\alpha$ be of homogeneous degree is inessential; the reader will easily check throughout the article that $\alpha$ may be broken into its homogeneous components in all of our results.} The notation $(X,\alpha)$ will always refer to such a pair for the rest of this article. The class $\alpha$ provides extra data which singles out special types of submanifolds, called \emph{exact}.

\begin{defn}
	We will call an embedded oriented hypersurface $Y \subset (X,\alpha)$ \textbf{exact} if $\alpha|_{Y} = 0 \in H^r(Y;\R)$.
\end{defn}

\begin{convn}
	When coefficients are not written, homology and cohomology are taken with integral coefficients, i.e. $H_*(X) = H_*(X;\Z)$ and $H^*(X) = H^*(X;\Z)$.
\end{convn}

\begin{convn}
	Henceforth, we shall always assume without stating that all submanifolds are oriented, so that $[Y] \in H_*(X)$ is a well-defined homology class.
\end{convn}

This article is concerned primarily with the following question.

\begin{quest} \label{quest:main}
	For a given homology class $Z \in H_{N-k}(X)$, when does there exist (or not exist) an embedded exact submanifold $Y^{N-k} \subset (X,\alpha)$ representing $Z$, i.e. with $[Y] = Z$?
\end{quest}

In this article, we find a sequence of obstructions to the existence of an embedded exact submanifold representing a given homology class $Z$. These obstructions arise by studying a certain formal deformation problem for the de Rham complex. Indeed, an exact submanifold yields a geometric deformation, which in turn implies that all of the formal algebraic obstructions must vanish. The algebraic side of the computation is the subject of Section \ref{sec:deform}, the results of which we know discuss.

Assume that $k$ is an odd integer. Fixed a closed differential form $\omega \in \Omega^*(X)$ and also a closed differential $k$-form $\eta \in \Omega^k(X)$. For $0 \leq L \leq \infty$ (NB: including the case of $L=\infty$), we may ask whether there exists a sequence of differential forms $\omega_1,\ldots,\omega_L \in \Omega^*(X)$ such that, setting $\omega_0 := \omega$, we have
$$d\omega_{\ell+1} = \eta \wedge \omega_\ell, \qquad \mathrm{for~all} ~ 0 \leq \ell < L.$$
It turns out that the existence of such a sequence depends only upon  the cohomology classes $[\omega] \in H^*(X;\R)$ and $[\eta] \in H^k(X;\R)$ as opposed to the specific differential forms $\omega$ and $\eta$. We say that $[\omega]$ is \textbf{$L$-jet deformable along $[\eta]$} if such a sequence $\omega_1,\ldots,\omega_L$ exists. This terminology, as well as the independence of representative forms, is clarified by the equivalence of this definition with an alternate one, Definition \ref{def:L-jet_deform}, as proved in Proposition \ref{prop:L-jet_equiv}. With this convenient notation, our main theorem is the following.\\

\begin{thm}\label{thm:main}
	Suppose we have a pair $(X^N,\alpha)$ consisting of a closed oriented manifold $X$ of dimension $N$ and a cohomology class $\alpha \in H^r(X;\R)$. Suppose an integral homology class $Z \in H_{N-k}(X)$ is represented by an exact submanifold, and by abuse of notation, let $PD(Z) \in H^k(X;\R)$ denote the image of the (integral) Poincar\'e dual class $PD(Z) \in H^k(X)$ under the map $H^k(X) \rightarrow H^k(X;\R)$.
	\begin{itemize}
		\item If $k$ is odd, then $\alpha$ is $\infty$-jet deformable in the direction of the $PD(Z)$.
		\item If $r$ is odd, then $PD(Z)$ is $\infty$-jet deformable in the direction of $\alpha$.
		\item Regardless of the parity of $k$ and $r$, we have $PD(Z) \wedge \alpha = 0 \in H^{r+k}(X;\R)$.\\
	\end{itemize}
\end{thm}

\begin{rmk} \label{rmk:Steenrod}
	It is not even the case that every homology class $Z \in H_{N-k}(X)$ is represented by a smooth submanifold of $X$, even without considering the data of $\alpha$. The study of realizability of homology classes is known as the Steenrod problem, and was initially studied by Thom \cite{Thom}. Despite the non-realizability of a general class, Thom proved that there always exsits some nonzero constant $c \in \Z$ such that $cZ$ is realized by a submanifold. Our obstructions are scale invariant, i.e. $\alpha$ is $L$-jet deformable in the direction of $PD(Z;\R)$ if and only if $\alpha$ is $L$-jet deformable in the direction of $cPD(Z;\R) = PD(cZ;\R)$. (See Proposition \ref{prop:scale_change}.) Hence, our obstructions in fact prevent every multiple $cZ$ for $c \neq 0$ from having an exact representative, and by Thom's result, this means our obstructions go beyond smooth realizability. See Proposition \ref{prop:1-to-infty} for an application of this line of reasoning.
\end{rmk}

Of consequence to the utility of Theorem \ref{thm:main} is an understanding of the computability of $\infty$-jet deformability. As an approximation, we may ask if a given class is $L$-jet deformable for all finite $L$. This may be thought of as a finite list of obstructions, the first one ($L=1$) recovering the condition $PD(Z) \wedge \alpha = 0$ that holds even if both $k$ and $r$ are even. Towards this purpose, we are in luck: for fixed $\mu \in H^k(X;\R)$ with $k$ odd, the vector subspace of cohomology classes in $H^*(X;\R)$ which are $L$-jet deformable in the direction of $\mu$ is explicitly computable via an inductive procedure on $L$, relying at each step on a finite dimensional linear algebra computation.\footnote{The computation is unfortunately nonlinear in $\mu$.} See Proposition \ref{prop:fin-dim} and the subsequent discussion for the details. Hence, we may utilize Theorem \ref{thm:main} in a very practical manner, inductively checking for each finite $L$ whether $\alpha$ is $L$-jet deformable along $PD(Z)$ if $k$ is odd, or vice versa if $r$ is odd.

In the final section, Section \ref{sec:symp_app}, we provide a number of applications in the setting of symplectic geometry, and especially the case $k=1$. If $(X,\omega)$ is a closed symplectic manifold, then $[\omega] \neq 0 \in H^2(X;\R)$, and hence the $\infty$-jet deformability of $[\omega]$ occurring in the first bullet point of Theorem \ref{thm:main} provides non-trivial obstructions to the existence of exact submanifolds in certain homology classes. Even in the case of $k=1$, we recover a number of symplectic manifolds which admit no non-separating exact hypersurfaces:
\begin{itemize}
	\item K\"ahler manifolds of dimension $\geq 4$ (first obstruction, see Corollary \ref{cor:Kaehler})
	\item Symplectically uniruled manifolds (first obstruction, see Theorem \ref{thm:symp_uniruled})
	\item The Kodaira--Thurston manifold (second obstruction, see Theorem \ref{thm:KT})
\end{itemize}
As a corollary, none of these manifolds admit non-separating contact-type surfaces, a particular type of exact hypersurface well-suited to pseudo-holomorphic curve theory, typically in the form of neck-stretching arguments in symplectic field theory (see \cite{EGH,BEHWZ}). The second statement follows as a corollary of work of Wendl \cite{Wendl}, who himself proved that symplectically uniruled manifolds do not admit non-separating contact-type hypersurfaces (though did not observe the version with exact hypersurfaces). The first and third results appear to be completely new in the literature.

\subsection*{Acknowledgments}

My initial interest in this topic stemmed from conversations with Baptiste Chantraine about Lee classes in locally conformal symplectic geometry. His explanation of work of Apostolov and Dloussky \cite{AD} led immediately to the conclusion of Theorem \ref{thm:Kaehler}, and in due course to the development of this article. A conversation with Rohil Prasad broadened the scope of this paper from just the symplectic setting to the more general setting presented here. Dennis Sullivan convinced me that the equations $d\omega_{\ell+1} = \eta \wedge \omega_{\ell}$ were interesting in their own right and provided further motivation to finish this article. (Relations to rational homotopy theory are relegated to future work.) Francesco Lin provided continual interest and encouragement, and conversations with Mike Sullivan and Umut Varolgunes at the right time convinced me to wrap up the paper. I am thankful for all of their guidance.

\section{Exactness is not homological}

Before proving our main results, let us get a feel for Question \ref{quest:main} from a few examples. For many cases of $(X,\alpha)$ and $Z$, an answer is elementary and completely satisfactory. Recall that $\dim X = N$, $|\alpha| = r$, and $|Z| = N-k$. We may clearly assume for degree reasons that $0 \leq k \leq N$ and $0 \leq r \leq N-k$. Let us consider some edge cases.

\begin{exam}
	If $\alpha = 0$, then every submanifold is exact. Correspondingly, notice that all of the obstructions of Theorem \ref{thm:main} clearly vanish.
\end{exam}

\begin{exam}
	The class $Z = 0$ is always represented by an embedded exact submanifold of any dimension, e.g. by any empty submanifold, or less trivially by any small trivial sphere. So we may focus our attention to the case $Z \neq 0$. Notice again that the obstructions of Theorem \ref{thm:main} clearly vanish.
\end{exam}

\begin{exam} \label{exam:top_codim}
	Consider the case $r=N-k$. Then exactness of a connected submanifold $Y^{N-k}$ is dependent only upon its homology class. That is, any $Y$ with $Y = [Z]$ is exact if and only if
	$$\langle \alpha, [Y] \rangle = 0,$$
	where the pairing is the standard one between homology and cohomology classes of the same degree. Comparing to Theorem \ref{thm:main}, we recover the obstruction
	$$\alpha \wedge PD(Z) = PD(\langle \alpha, [Y]\rangle).$$
	However, in the case where $k$ or $r$ is odd, it appears as though $\infty$-jet deformability is irrelevant to this geometric reasoning. See Proposition \ref{prop:1-to-infty} which exploits this point.
\end{exam}

Once we move away from these examples, the question becomes more subtle, as the exactness of a hypersurface is no longer determined by its homology class alone. The following is a silly example.

\begin{exam}\label{exam:r=0}
	Suppose $\alpha \in H^0(X;\R)$ is homogeneous of degree $0$. We may identify $\alpha$ with a locally constant function on $X$, from which $X$ decomposes as
	$$X = \alpha^{-1}(0) \sqcup (X \setminus \alpha^{-1}(0))$$
	where each of the two terms is a union of components of $X$. A submanifold is exact if and only if it is contained in $\alpha^{-1}(0)$. Note that even though this condition of exactness is relatively straightforward, the exactness of a submanifold is not determined by its homology class.
\end{exam}

The situation from $r=0$ generalizes to Proposition \ref{prop:non-homological} below, which tells us that exactness of a given submanifold $Y \subset X$ depends in general upon its specific embedding, and not just on its homology class $[Y] \in H_{N-k}(X)$. It is for this reason that an answer to Question \ref{quest:main} is subtle.

\begin{prop} \label{prop:non-homological}
	For any $N \geq 3$, $1 \leq k \leq N-2$, and $1 \leq r \leq N-k-1$, there exists a pair $(X,\alpha)$ consisting of a closed connected oriented $N$-manifold $X$ and a cohomology class $\alpha \in H^r(X;\R)$, together with a pair of $(N-k)$-dimensional embedded submanifolds $Y_1,Y_2 \subset X$, such that $[Y_1] = [Y_2] \in H_{N-k}(X)$ with $Y_1$ exact but $Y_2$ not exact.
\end{prop}

\begin{proof}
	Consider the manifold with boundary $C = (S^r \times D^{N-k-r+1}) \setminus D^{N-k+1}$, thought of as a cobordism from $W_1 = S^{N-k}$ to $W_2 = S^{r} \times S^{N-k-r}$. We may take the double of $C$, i.e.
	$$\D C := C \cup_{W_1 \sqcup W_2} C.$$
	Consider on $C$ the restriction of a generating cohomology class $\beta_{\Z} \in H^r(S^r \times \D^{N-k-r+1}) \cong \Z$, and let $\beta$ be its real image. By Mayer--Vietoris, we see that we obtain a class $\wt{\beta} \in H^r(\D C; \R)$ which restricts to $\beta$ on each copy of $C$. Finally, if $M^{k-1}$ is a closed oriented manifold, take $X = M \times \D C$, and let $\alpha \in H^r(M \times \D C; \R)$ be the pull-back of $\wt{\beta}$ under projection. Letting $Y_j = M \times W_j$, the conclusion now holds by the K\"unneth theorem. That is, since $\alpha$ is pulled back, we see that we may identify $\alpha|_{Y_j}$ with
	$$\beta|_{W_j} \in H^r(W_j;\R) \cong H^0(M;\R) \otimes H^r(W_j;\R)  \stackrel{\text{K\"unneth}}{\cong} H^r(Y_j;\R).$$
	Clearly $\beta|_{W_1} = 0$ whereas $\beta|_{W_2} \neq 0$, and hence the result follows.
\end{proof}

\begin{rmk}
	In the proof of Proposition \ref{prop:non-homological}, it is essential that $Y_1$ and $Y_2$ are non-homotopic even though they are homologous, since if $Y_1$ and $Y_2$ are homotopic, then $Y_1$ is exact if and only if $Y_2$ is exact.
\end{rmk}

\section{$L$-jet deformability} \label{sec:deform}

\subsection{Motivation and formal deformations}

The geometric question (Question \ref{quest:main}) and partial obstructive answer (Theorem \ref{thm:main}) at the heart of this paper are very much related to the algebraic study of deformations of closed elements in the de Rham complex as we simultaneously twist the differential.\footnote{Indeed, the initial motivation for this work arose from a desire to understand the study of deformations of locally conformal symplectic structures (LCS) upon deformation of the Lee class, as in work of Apostolov and Dloussky \cite{AD}. This work began from the realization that exact hypersurfaces in homology class $Z$ yield deformations of the Lee class in the direction of $PD(Z)$.} We discuss the basic algebraic set-up in this section; a more general discussion appears in Section \ref{ssec:general}.

\begin{defn}
	Given a closed $1$-form $\eta \in \scr{Z}^1(X)$, one may form the \textbf{Lichnerowicz--de Rham differential} $$d_{\eta} \colon \Omega^*(X) \rightarrow \Omega^{*+1}(X)$$ given by $d_{\eta}\omega := d\omega - \eta \wedge \omega$.
\end{defn}

The Lichnerowicz--de Rham differential satisfies $d_{\eta}^2 = 0$ (precisely because $d\eta = 0$), and hence gives a new `twisted' cochain complex structure on $\Omega^*(X)$. The basic algebraic question we study is as follows:
\begin{quest}
	Suppose $\omega_0 \in \scr{Z}^*(X)$ is a closed differential form (with respect to $d$). Does there exist a family of closed differential forms $\omega(t) \in \scr{Z}^*(X)$ with $\omega(0) = \omega_0$ such that $d_{t\eta}\omega_t = 0$?
\end{quest}
At the infinitesimal level, one may find obstructions by taking the Taylor series around $t=0$. That is to say that we are looking for
$$\omega(t) = \sum_{j=0}^{\infty} \omega_j t^j$$
with constant coefficient $\omega_0$ the same as specified in our question such that
$$d_{t\eta} \omega(t) = 0.$$
Separating out power by power, we arrive at the equations $d\omega_{j+1} = \eta \wedge \omega_j$. In other words, we arrive at our notion of $\infty$-jet deformability from the introduction. Algebraically, what we have done is to form the \textbf{formal Lichnerowicz--de Rham complex}
$$\scr{D}^{\infty}_{\eta}(X) := (\Omega^*(X) \otimes_{\R} \R\llbracket t\rrbracket, d_{t\eta})$$
where
$$d_{t\eta}(\omega \otimes p(t)) := d\omega \otimes p(t) - (\eta \wedge \omega) \otimes tp(t).$$

Suppose now that $\eta \in \scr{Z}^k(X)$ is a closed differential form of homogeneous degree $k$. We may still try to form the same formal complex. If we work too na\"ively, we note that if $k$ is even, then
$$(d_{t\eta})^2\omega = t^2\eta \wedge \eta \wedge \omega$$
is not necessarily $0$ even though $d\eta = 0$. This is fixed in this parity by setting $t^2=0$. In fact, this arises rather naturally by considering grading. Since we desire a cochain complex, we need that $d_{t\eta}$ has degree $1$, and hence take $|t|=1-k$. The Koszul sign rule for graded commutativity in turn implies that $t^2=0$ if $k$ is odd.

Hence, we arrive at the following dichotomy for the formal Lichnerowicz--de Rham complex:
$$\scr{D}^{\infty}_{\eta}(X) :=
\begin{cases}
	(\Omega^*(X) \otimes_{\R} \R\llbracket t\rrbracket, d_{t\eta}),& k~\mathrm{odd}\\
	(\Omega^*(X) \otimes_{\R} \R[t]/(t^2), d_{t\eta}),& k~\mathrm{even}
\end{cases}$$
We see that for $k$ even, an element $\omega(t)$ may be written just as $\omega_0 + t \omega_1$, and it is closed precisely if $d\omega_1 = \eta \wedge \omega_0$, with no higher $\omega_j$ required. It is for this reason that the statement of Theorem \ref{thm:main} comes with parity conditions. We may therefore freely restrict to $k$ odd for the rest of the section.\footnote{Everything we write is legitimate for all $k$ so long as one remembers $t^2=0$ in the case of $k$ even.}

\subsection{$L$-jet approximations}

We may take $L$-jet approximations to our formal Lichnerowicz--de Rham complex by modding out by the ideal $(t^{L+1})$. Explicitly, if we fix a closed $k$-form $\eta \in \scr{Z}^k(X)$ and an integer $L \geq 0$, we form the $L$-jet approximation
\begin{eqnarray*}
\scr{D}^L_{\eta}(X) &:=& \scr{D}^{\infty}_{\eta}(X) \otimes_{\R} \quot{\R\llbracket t\rrbracket}{(t^{L+1})} \\
	&=& \left(\Omega^*(X) \otimes \quot{\R\llbracket t\rrbracket}{(t^{L+1})},d_{t\eta}\right)
\end{eqnarray*}
where the formal variable $t$ is again given degree $1-k$, and where $d^L_{t\eta}$ indicates applying $d_{t\eta}$ but modding out by $t^{L+1}$. Indeed, it is clear that $d^L_{t\eta}$ descends as such since it preserves the ideal $(t^{L+1})$. The fact that $d_{t\eta}$ was a differential descends also to $d^L_{t\eta}$, i.e. $(d^L_{t\eta})^2 = 0$.

\begin{exam} \label{exam:no_deform}
	In the case $L=0$, we have simply
	$$\scr{D}^0_{\eta}(X) = (\Omega^*(X),d),$$
	i.e. there is no formal deformation of the de Rham complex at all.
\end{exam}

We note that it is natural to include the case $L=\infty$, in which case we recover the complex $\scr{D}^{\infty}_{\eta}(X)$. We shall henceforth allow $L = \infty$ where possible.

The complexes $\scr{D}^L_{\eta}(X)$ come with a number of maps between them. We will define them first as graded linear maps, though they all actually intertwine the corresponding differentials, and hence yield maps of cochain complexes, which we state as part of Proposition \ref{prop:cochain_map_properties}. The three types of maps we consider are as follows:

\noindent \textbf{1. Truncation maps:} For each $0 \leq L_1 \leq L_2 \leq \infty$, we obtain a truncation map
$$\tau^{L_2,L_1}_{\eta} \colon \scr{D}_{\eta}^{L_2}(X) \rightarrow \scr{D}_{\eta}^{L_1}(X)$$
induced by the quotient maps $\R\llbracket t \rrbracket/(t^{L_2+1}) \rightarrow \R\llbracket t \rrbracket/(t^{L_1+1})$.\\

\noindent \textbf{2. Gauge change isomorphisms:} If $\eta' = \eta+dg$ is another cohomologous closed $k$-form, with $g$ a fixed $(k-1)$-form, then we obtain a gauge change map
$$\Phi^L_{\eta;g} \colon \scr{D}^L_{\eta}(X) \rightarrow \scr{D}^L_{\eta'}(X)$$
given by (formal) multiplication by
$$e^{tg} := \sum_{j=0}^{L}\frac{g^j}{j!}t^j \in C^{\infty}(X) \otimes \R\llbracket t \rrbracket/(t^{L+1}).$$
In fact, these are isomorphisms with inverse given by multiplying by $e^{-tg}$, i.e. the map $\Phi^L_{\eta';-g}$.

\begin{rmk}
	Even if $\eta = \eta'$, we may take $g$ to be a nonzero closed $(k-1)$-form and the resulting automorphism $\Phi^L_{\eta;g}$ will not be the identity.
\end{rmk}

\noindent\textbf{3. Scale change isomorphisms:} If $c \in \R^*$, then we obtain a scale change map
$$s^L_{\eta;c} \colon \scr{D}^L_{\eta}(X) \rightarrow \scr{D}^L_{c\eta}(X)$$
given by sending $t \mapsto t/c$, i.e.
$$s^L_{\eta;c}\left(\sum_{j=0}^{L} \alpha_j t^j\right) = \sum_{j=0}^{L} \frac{\alpha_j}{c^j}t^j.$$
This is in fact an isomorphism because it has inverse $s^L_{c\eta;1/c}$.\\

\begin{prop} \label{prop:cochain_map_properties}
	The truncation maps, gauge change isomorphisms, and scale change isomorphisms are all (degree zero) maps of cochain complexes which satisfy the following properties:
	\begin{itemize}
		\item \textbf{Cocycle conditions of truncation maps:} $$\tau^{L,L}_{\eta} = \mathrm{id}, \qquad \qquad \tau^{L_2,L_1}_{\eta} \circ \tau^{L_3,L_2}_{\eta} = \tau^{L_3,L_1}_{\eta}.$$
		\item \textbf{Cocycle conditions for gauge change:}
		$$\Phi^L_{\eta;0} = \mathrm{id}, \qquad \qquad \Phi^L_{\eta;g_1+g_2} = \Phi^L_{\eta+dg_1;g_2} \circ \Phi^L_{\eta;g_1}$$
		\item \textbf{Cocycle conditions for scale change:}
		$$s^L_{\eta;1} = \mathrm{id}, \qquad \qquad s^L_{c_1\eta;c_2} \circ s^L_{\eta;c_1} = s^L_{\eta;c_1c_2}.$$
		\item \textbf{Commutativity of truncation and gauge change:}
		$$\tau_{\eta+dg}^{L_2,L_1} \circ \Phi_{\eta;g}^{L_2} = \Phi_{\eta;g}^{L_1} \circ \tau_{\eta}^{L_2,L_1}$$
		\item \textbf{Commutativity of truncation and scale change:}
		$$\tau_{c\eta}^{L_2,L_1} \circ s^{L_2}_{\eta;c} = s^{L_1}_{\eta;c} \circ \tau_{\eta}^{L_2,L_1}$$
		\item \textbf{Commutativity of gauge change and scale change:}
		$$\Phi_{c\eta;c g}^L \circ s^{L}_{\eta;c} = s^{L}_{\eta+dg;c} \circ \Phi^L_{\eta;g}$$
	\end{itemize}
\end{prop}
\begin{proof}[Proof, mostly left as an exercise for the reader]
	Let us prove that gauge change is a map of cochain complexes; in other words that
	$$d^L_{t(\eta+dg)} \circ \Phi^L_{\eta;g} = \Phi^L_{\eta;g} \circ d^L_{t\eta}.$$
	We indeed check this is the case:
	\begin{eqnarray*}
		d^L_{t(\eta+dg)}\Phi^L_{\eta;g}\alpha &=& d^L_{t(\eta+dg)}(e^{tg}\alpha) \\
		&=& d(e^{tg}\alpha) - t(\eta+dg)e^{tg}\alpha (\mathrm{mod}~t^{L+1}) \\
		&=& e^{tg} \cdot tdg \wedge \alpha + e^{tg}d\alpha - t(\eta+dg)e^{tg}\alpha (\mathrm{mod}~t^{L+1}) \\
		&=& e^{tg}(d\alpha - t\eta \wedge \alpha) (\mathrm{mod}~t^{L+1}) \\
		&=& \Phi^L_{\eta;g}d^L_{t\eta}\alpha
	\end{eqnarray*}
	The other properties are similarly easily verified.	
\end{proof}

\begin{rmk}
	In order to define the gauge change and scale change isomorphisms, we only needed to define them in the case of $L = \infty$. Indeed, these maps on $\scr{D}^{\infty}_{\eta}(X)$ preserve the ideals $(t^{L+1})$, and so they descend uniquely to each $\scr{D}^L_{\eta}(X)$ in such a way to commute with the truncation maps.
\end{rmk}

Maps of cochain complexes always preserve exact elements. In fact, for the truncation maps, we record the following useful lemma.

\begin{lem} \label{lem:surjective_exact}
	Let $\scr{B}^*(\scr{D}_{\eta}^{L}(X))$ represent the exact elements. Then for $L_1 \leq L_2$, the truncations
	$$\tau_{\eta}^{L_2,L_1} \colon \scr{B}^*(\scr{D}_{\eta}^{L_2}(X)) \rightarrow \scr{B}^*(\scr{D}_{\eta}^{L_1}(X))$$
	are surjective.
\end{lem}
\begin{proof}
	Suppose $\alpha = d^{L_1}_{t\eta}\beta \in \scr{B}^*(\scr{D}_{\eta}^{L_1}(X))$. We have $\tau_{\eta}^{L_2,L_1}$ is surjective at the cochain level, so we may lift $\beta$ to some $\wt{\beta} \in \scr{D}_{\eta}^{L_2}(X)$ with
	$$\tau_{\eta}^{L_2,L_1}(\wt{\beta}) = \beta.$$
	But then since $\tau_{\eta}^{L_2,L_1}$ is a map of cochain complexes, we find
	$$\tau_{\eta}^{L_2,L_1}(d_{t\eta}^{L_2}\wt{\beta}) = d_{t\eta}^{L_1}\left(\tau_{\eta}^{L_2,L_1}(\wt{\beta})\right) = d_{t\eta}^{L_1}\beta = \alpha,$$
	and so we obtain the result.
\end{proof}

Because our maps are all maps of cochain complexes, they descend to cohomology. We will use the same notation for the maps acting at the cochain level versus at cohomology, though the meaning will be clear from context.

In the case $L=0$, so that $\scr{D}^0_{\eta}(X) = (\Omega^*(X),d)$ as in Example \ref{exam:no_deform}, the gauge change isomorphism $\Phi^0_{\eta;g}$, considered as an automorphism of $(\Omega^*(X),d)$, is just the identity, regardless of $\eta$ and $g$. Combining this with Proposition \ref{prop:cochain_map_properties} we see the following diagram commutes for each $L$, $\eta$, and $g$:
\[\xymatrix{H^*(\scr{D}_{\eta}^L(X)) \ar^-{\tau^{L,0}_{\eta}}[r] \ar[d]_{\Phi^L_{\eta;g}} & H^*(\scr{D}_{\eta}^0(X)) \ar[d]_{\Phi^0_{\eta;g}} \ar@{=}[r] & H^*(X;\R) \ar@{=}[d] \\
H^*(\scr{D}_{\eta+dg}^L(X)) \ar^-{\tau^{L,0}_{\eta+dg}}[r] & H^*(\scr{D}_{\eta+dg}^0(X)) \ar@{=}[r] & H^*(X;\R)}\]
It follows, since the left vertical arrow is an isomorphism, that the images of $\tau_{\eta}^{L,0}$ and $\tau_{\eta+dg}^{L,0}$, as subspaces of $H^*(X;\R)$, are identical. We denote this subspace by
$$V^L_{[\eta]} := \mathrm{Im}(\tau_{\eta}^{L,0} \colon H^*(\scr{D}_{\eta}^L(X)) \rightarrow H^*(X;\R)),$$
where we now have implicit in the notation the fact that this subspace depends only upon the cohomology class $[\eta] \in H^k(X;\R)$, and not on the specific representative closed $1$-form $\eta$.

\begin{defn} \label{def:L-jet_deform}
	Suppose $\alpha \in H^*(X;\R)$ and $\mu \in H^k(X;\R)$. Then $\alpha$ is said to be \textbf{$L$-jet deformable along $\mu$} if $\alpha \in V^L_{\mu}$.
\end{defn}

Let us unwind this definition to make sure it agrees with the introduction.

\begin{prop} \label{prop:L-jet_equiv}
	Suppose $\omega \in \Omega^r(X)$ is a closed differential $r$-form and $\eta \in \scr{Z}^k(X)$ is closed differential $k$-form. Then $[\omega]$ is $L$-jet deformable along $[\eta]$ if and only if there exists a sequence of differential form $\omega_1,\ldots,\omega_L$ such that, setting $\omega_0 = \omega$, we have
	$$d\omega_{\ell+1} = \eta \wedge \omega_\ell~\text{for all}~ 0 \leq \ell < L$$
\end{prop}
\begin{proof}
	From the formula for $d^{L}_{t\eta}$, we find that an element
	$$\sum_{j=0}^{L}\omega_jt^j \in \scr{D}^L_{\eta}(X)$$
	is closed if and only if $d\omega_0 = 0$ and $d\omega_{\ell+1} = \eta \wedge \omega_\ell$ for all $0 \leq \ell < L$, and so the reverse implication is immediate. For the forward implication, suppose $[\omega] \in V^L_{[\eta]}$. Then there is some closed element
	$$\alpha:=\sum_{j=0}^{L}\alpha_jt^j \in \scr{D}^L_{\eta}(X)$$
	such that
	$$[\omega] = [\tau^{L,0}_{\eta}(\alpha)] = [\alpha_0].$$
	By Lemma \ref{lem:surjective_exact}, we have that the exact form $\omega-\alpha_0$ may be written as
	$$\omega-\alpha_0 =\tau_{\eta}^{L,0}(\beta)$$
	for some exact $\beta \in \scr{D}^L_{\eta}(X)$. Hence, $\alpha+\beta$ is still closed, and
	$$\omega = \tau_{\eta}^{L,0}(\alpha+\beta)$$
	at the cochain level, not just on cohomology. The coefficients of $\alpha+\beta$ yield the desired sequence $\omega_0 = \omega, \omega_1, \ldots,\omega_L$.
\end{proof}

We record one more property from our algebraic package: the vector space $V^L_{\mu}$ in fact only depends on the conformal class of $\mu$ as opposed to the specific cohomology class.

\begin{prop}\label{prop:scale_change}
	For $c \neq 0$ a nonzero constant and $\mu \in H^k(X;\R)$,
	$$V^L_{\mu} = V^L_{c\mu}.$$
	In other words, a cohomology class $\alpha \in H^*(X;\R)$ is $L$-jet deformable in the direction of $\mu$ if and only if it is $L$-jet deformable in the direction of $c\mu$.
\end{prop}
\begin{proof}
	Following the same reasoning as before but switching gauge change for scale change, the commutative diagrams
	\[\xymatrix{H^*(\scr{D}_{\eta}^L(X)) \ar^-{\tau^{L,0}_{\eta}}[r] \ar[d]_{s^L_{\eta;c}} & H^*(\scr{D}_{\eta}^0(X)) \ar[d]_{s^0_{\eta;c}} \ar@{=}[r]^-{\sim} & H^*(X;\R) \ar@{=}[d] \\
		H^*(\scr{D}_{c\eta}^L(X)) \ar^-{\tau^{L,0}_{c\eta}}[r] & H^*(\scr{D}_{c\eta}^0(X)) \ar@{=}[r]^-{\sim} & H^*(X;\R)}\]
	(where $[\eta] = \mu \in H^k(X;\R)$) imply that
	$$V^L_{\mu} = \mathrm{Im}(\tau^{L,0}_{\eta}) = \mathrm{Im}(\tau^{L,0}_{c\eta}) = V^L_{c\mu}.$$
	
\end{proof}

\subsection{Computability of $L$-jet deformability}

It is natural to try to understand the vector spaces $V^L_{[\eta]}$ for each $L < \infty$. We begin by highlighting that it suffices to study the space in each degree. That is, since each $\tau^{L,0}_{\eta}$ preserves degree, we have that
$$V^L_{[\eta]} = \bigoplus_{r = 0}^{\infty} \tau^{L,0}_{\eta}(H^r(\scr{D}^L_{\eta}(X))) = \bigoplus_{r = 0}^{\infty} (V^L_{[\eta]} \cap H^r(X;\R)).$$
It hence suffices to study the subspaces
$$V^{L,r}_{[\eta]} := \tau^{L,0}_{\eta}(H^r(\scr{D}^L_{\eta}(X))) = V^L_{[\eta]} \cap H^r(X;\R)$$
for each $r \geq 0$. In order to understand them, we shall find a finite spanning set for each $H^r(\scr{D}^L_{\eta}(X))$, the image of which under $\tau^{L,0}_{\eta}$ yields a spanning set for $V^{L,r}_{[\eta]}$.

\begin{prop} \label{prop:fin-dim} 
	For each $L < \infty$ and $r \geq 0$, we have
	$$\dim H^r(\scr{D}_{\eta}^{L}(X)) \leq \sum_{j=0}^{L} \dim H^{r+j(k-1)}(X;\R).$$
\end{prop}
\begin{proof}
	The first inequality is obvious. For each $0 < L < \infty$, let $\Omega^*(X)[-L(k-1)]$ denote the usual de Rham complex shifted in degree so that e.g. the degree zero part of this new complex is given by $\Omega^{L(k-1)}(X)$. The map
	$$\Psi^L_{\eta} \colon \Omega^*(X)[-L(k-1)] \rightarrow \scr{D}_{\eta}^L(X)$$
	given by
	$$\Psi^L_{\eta}(\omega) = \omega \cdot t^L~(\mathrm{mod}~t^{L+1})$$
	is a map of cochain complexes, since it preserves degree and intertwines the differential:
	$$\Psi^L_{\eta}(d\omega) = d\omega \cdot t^L = d^L_{t\eta}(\omega \cdot t^L) = d^L_{t\eta}(\Psi^L_{\eta}(\omega)).$$
	Furthermore, it fits into a short exact sequence of cochain complexes
	$$0 \rightarrow \Omega^*(X)[-L(k-1)] \xrightarrow{\Psi^{L}_{\eta}} \scr{D}^{L}_{\eta}(X) \xrightarrow{\tau^{L,L-1}_{\eta}} D^{L-1}_{\eta}(X) \rightarrow 0.$$
	Hence, the long exact sequence in cohomology reads
	$$\cdots \rightarrow H^{r+L(k-1)}(X;\R) \rightarrow H^r(\scr{D}^{L}_{\eta}(X)) \rightarrow H^r(\scr{D}^{L-1}_{\eta}(X)) \rightarrow \cdots$$
	from which it follows that
	$$\dim H^r(\scr{D}^{L}_{\eta}(X)) \leq \dim H^{r+L(k-1)}(X;\R) + \dim H^r(\scr{D}^{L-1}_{\eta}(X)).$$
	The result now follows by induction, where the base case $L=0$ is clear.
\end{proof}

The proof of Proposition \ref{prop:fin-dim} provides an algorithm for computing a spanning set for the vector spaces $H^r(\scr{D}^{L}_{\eta}(X))$, and hence $V^{L,r}_{[\eta]}$ after applying $\tau^{L,0}_{\eta}$, via induction on $L$, as we now explain. Let us write
$$\scr{C}_{\eta}^{L,r}(X) := \mathrm{Im}(H^r(\scr{D}_{\eta}^{L+1}(X))\xrightarrow{\tau^{L+1,L}_{\eta}} H^r(\scr{D}_{\eta}^{L}(X))),$$
so that the long exact sequence considered in the proof above may be truncated to an exact sequence of the form
$$\cdots \rightarrow H^{r+(L+1)(k-1)}(X;\R) \xrightarrow{\Psi^{L+1}_{\eta}} H^r(\scr{D}^{L+1}_{\eta}(X)) \xtwoheadrightarrow{\tau^{L+1,L}_{\eta}} \scr{C}^{L,r}_{\eta}(X) \rightarrow 0.$$
Assume by induction that we have a finite spanning set
$$\{[v_1],\ldots, [v_M]\} \subset H^r(\scr{D}_{\eta}^{L}(X)),$$
where each $v_\ell = \sum_{j=0}^{L}v_{j\ell}t^j \in \scr{D}_{\eta}^{L}(X)$ for some $v_{j\ell} \in \Omega^r(X)$. 

\begin{lem}
	The form $\eta \wedge \sum_{\ell=0}^{M}c_kv_{L\ell}$ is exact if and only if $\sum_{\ell=0}^{M}c_\ell[v_\ell] \in \scr{C}_{\eta}^{L,r}(X)$.
\end{lem}
\begin{proof}
	If the wedge product is exact, so that we may write
	$$\eta \wedge \sum_{\ell=0}^{M}c_\ell v_{L\ell} = d\alpha,$$
	then
	$$\sum_{\ell=0}^{M} c_\ell[v_\ell] = \left[\tau^{L+1,L}_{\eta}\left(\sum_{j=0}^{L}\left(\sum_{\ell=0}^{M}c_\ell v_{j\ell}\right) + \alpha \cdot t^{L+1}\right)\right] \in \scr{C}_{\eta}^{L,r}(X).$$

	Conversely, if $\sum_{\ell=0}^{M(L)}c_\ell[v_\ell] \in \scr{C}^{L,r}_{\eta}(X)$, then at the cochain level, we have 
	$$\sum_{\ell=0}^{M} c_\ell v_\ell = \alpha + \tau^{L+1,L}_{\eta}\beta$$
	where $\alpha \in \scr{D}^L_{\eta}(X)$ is exact, $\beta \in \scr{D}^{L+1}_{\eta}(X)$ is closed, and $\alpha$ and $\beta$ have degree $r$. By Lemma \ref{lem:surjective_exact}, $\alpha$ is itself the image of an exact element under $\tau^{L+1,L}_{\eta}$; hence,
	$$\sum_{\ell=0}^{M} c_\ell v_\ell = \tau^{L+1,L}_{\eta}(\gamma),$$
	where $\gamma \in \scr{D}^{L+1}_{\eta}(X)$ is closed (of degree $r$), . Writing $\gamma = \sum_{j=0}^{L+1}\gamma_jt^j$, we have $\gamma_L = \sum_{k=0}^{M}c_\ell v_{L\ell}$. On the other hand, since $\gamma$ is closed, $$d\gamma_{L+1} = \eta \wedge \gamma_L = \eta \wedge \sum_{k=0}^{M}c_\ell v_{L\ell}$$ is exact.
\end{proof}

Thinking of the spanning set as a surjective map $\R^{M} \twoheadrightarrow H^r(\scr{D}^L_{\eta}(X))$, the lemma gives necessary and sufficient linear conditions on elements of $\R^{M}$ for the image to land in $\scr{C}^{L,r}_{\eta}$, and taking a basis for this subspace of $\R^{M}$ therefore yields a spanning set $[w_1],\ldots,[w_{M'}]$ for $\scr{C}^{L,r}_{\eta}$. For each $w_j$, we can choose a lift $[\wt{w}_j] \in H^r(\scr{D}^{L+1}_{\eta}(X)$. Then
$$H^r(\scr{D}_{\eta}^{L+1}(X) = \mathrm{Im}(\Psi^{L+1}_{\eta})+\langle[\wt{w}_1], \ldots, [\wt{w}_{M'}]\rangle$$
where the first summand clearly has a spanning set by taking the images of a basis of $H^r(X;\R)$. Hence, we have inductively constructed a spanning set for $H^r(\scr{D}^{L+1}_{\eta}(X))$ out of a spanning set for $H^r(\scr{D}^{L}_{\eta}(X))$.

\begin{rmk}
	It is harder in general to choose a \emph{basis} for each $H^r(\scr{D}^L_{\eta}(X))$, which would essentially require one to understand which elements of $H^{r+L(k-1)}(X;\R)$ generate equivalent elements under $\Psi^L_{\eta}$ for each $L$. This is a difficult computation in general unless we are in rather restrictive settings.
\end{rmk}

\begin{exam} \label{exam:nilmanifold}
	Let us compute $H^r(\scr{D}^L_{\eta}(X))$ in an explicit example. Consider the $4$-dimensional Lie group
	$$G_{\R} = H_{\R} \oplus \R,$$
	where $H_{\R}$ is the Heisenberg group
	$$H_{\R} = \left\{\begin{pmatrix}1 & a & b \\ 0 & 1 & c \\ 0 & 0 & 1\end{pmatrix} \middle| a,b,c \in \R\right\}.$$
	We note that requiring $a,b,c \in \Z$ yields a Lie subgroup $H_{\Z} \leq H_{\R}$. We may therefore take the quotient
	$$X = G_{\R}/G_{\Z} = (H_{\R}\oplus \R)/(H_{\Z} \oplus \Z) = H_{\R}/H_{\Z} \times \R/\Z.$$
	Because it is a closed nilmanifold, a classical result of Nomizu \cite{Nomizu} states that the inclusion of the subcomplex of left-invariant forms into the full de Rham dg-algebra is a quasi-isomorphism (meaning an isomorphism on cohomology). This subcomplex is computed as the Chevalley--Eilenberg dg-algebra $(\Lambda^*\mathfrak{g}^*,d)$. Explicitly, if $A,B,C,T$ are dual to $\partial_a,\partial_b,\partial_c,\partial_t$ (where $\partial_t$ spans the $\R$-direction), then the complex is generated by the relations
	$$dA = dB = dT = 0, \qquad dC = AB.$$
	(Note that we leave the wedge product out of the notation for convenience, so $AB$ really means $A \wedge B$, thinking of $A$ and $B$ as left-invariant forms.) Explicitly, writing cohomology as closed forms modulo exact forms, we have
	\begin{eqnarray*}
		H^0(X;\R)  &=& \mathrm{span}(1) \\
		H^1(X;\R) &=& \mathrm{span}(A,B,C,T)/\mathrm{span}(C) \\
		H^2(X;\R) &=& \mathrm{span}(AB,AC,AT,BC,BT)/\mathrm{span}(AB) \\
		H^3(X;\R) &=& \mathrm{span}(ABZ,ABT,ACT,BCT)/\mathrm{span}(ABT) \\
		H^4(X;\R) &=& \mathrm{span}(ABCT)
	\end{eqnarray*}
	With this out of the way, suppose we take the class $\eta = A$ (of degree $k=1$) and $r=2$ and compute the vector spaces $H^2(\scr{D}^L_{A}(X))$. We begin with the trivial case $L=0$:
	$$H^2(\scr{D}^0_{A}(X)) = H^2(X;\R).$$
	Take as a basis for $H^2(X;\R)$ the forms $AC,AT,BC,BT$, so that a general element of $H^2(\scr{D}^0_{A}(X)$ may be written as
	$$[\omega_0] = [\alpha_0AC+\beta_0AT+\gamma_0BC+\delta_0BT].$$
	We see that
	$$\eta \wedge \omega_0 = \gamma_0ABC + \delta_0ABT$$
	is exact if and only if $\gamma_0 = 0$, with primitive $\delta_0 CT$, i.e.
	$$d\left(\delta_0 CT\right) = \delta_0 ABT = \eta \wedge \omega_0.$$
	We hence find that
	$$H^2(\scr{D}^1_{A}(X)) = \left\{\left[\begin{matrix}(\alpha_0AC+\beta_0AT+\delta_0BT) \\ +(\delta_0CT + \alpha_1AC+\beta_1AT+\gamma_1BC+\delta_1BT)t \end{matrix}\right]\right\},$$
	where the terms $\alpha_1AC +\beta_1AT+\gamma_1BC+\delta_1BT$ in the second coordinate correspond to the image of $\Psi^1_A$. Extending further, we see that
	\begin{eqnarray*}\eta \wedge \omega_1 &=& A \cdot (\delta_0CT + \alpha_1AC+\beta_1AT+\gamma_1BC+\delta_1BT) \\
		&=& \delta_0ACT + \gamma_1ABC+\delta_1ABT.
	\end{eqnarray*}
	In order for this to be exact, we now need $\delta_0 = \gamma_1 = 0$. We find
	$$H^2(\scr{D}^2_{A}(X)) = \left\{\begin{bmatrix}(\alpha_0AC+\beta_0AT) \\  +(\alpha_1AC+\beta_1AT+\delta_1BT)t \\
	+(\delta_1CT + \alpha_2AC+\beta_2AT+\gamma_2BC+\delta_2BT)t^2\end{bmatrix}\right\}.$$
	Continuing on, it is not hard to see that
	$$H^2(\scr{D}^L_{A}(X)) = \left\{\begin{bmatrix}(\alpha_0AC+\beta_0AT) \\  +(\alpha_1AC+\beta_1AT)t \\
		+(\alpha_2AC+\beta_2AT)t^2 \\
		\vdots \\
		+(\alpha_{L-2}AC+\beta_{L-2}AT)t^{L-2} \\
		+(\alpha_{L-1}AC+\beta_{L-1}AT+\delta_{L-1}BT)t^{L-1}\\
		+(\delta_{L-1}CT + \alpha_LAC+\beta_LAT+\gamma_LBC+\delta_LBT)t^{L}\end{bmatrix}\right\}.$$
	In the case of $L = \infty$, we see that
	$$H^r(\scr{D}^{\infty}_A(X)) =  \left\{\begin{bmatrix}(\alpha_0AC+\beta_0AT) \\  +(\alpha_1AC+\beta_1AT)t \\
		+(\alpha_2AC+\beta_2AT)t^2 \\
		\vdots \end{bmatrix}\right\}.$$
	We find in particular that for every $2 \leq L \leq \infty$,
	$$V^{L,2}_A(X) = \mathrm{span}\langle [AC],[AT] \rangle,$$
	whereas
	$$V^{1,2}_A(X) = \mathrm{span}\langle [AC],[AT],[BT] \rangle.$$
	For example, $[BT]$ is $1$-jet deformable, but not $2$-jet deformable.

	It is worth remarking that the expressions for $H^r(\scr{D}^L_A(X))$ are written in terms of spanning sets, not bases. For example, note that
	$$\Psi^2_{A}([AC]) = [ACt^2] = [d_{tA}(-B+tC)] = 0,$$
	and hence the parameter $\alpha_2$ is extraneous in our expression for $H^2(\scr{D}^2_{A}(\Omega^*(X),d))$.
\end{exam}

\subsection{Proof of Main Theorem}

We have all the tools we need to prove Theorem \ref{thm:main} from the introduction, our main result in this article.

\begin{proof}[Proof of Theorem \ref{thm:main}]
	We are given $Y \subset X$ is an exact submanifold, cooriented (as per our implicit assumption that $Y$ is oriented from the introduction), and representing the class $Z \in H_{N-k}(X)$. Recall, e.g. in the book of Bott and Tu \cite[Section 6]{Bott-Tu} that $PD(Z)$ may be represented by a differential form giving the Thom class of the normal bundle, i.e. by a closed $k$-form $\eta \in \scr{Z}^k(X)$ compactly supported on a tubular neighborhood $U$ of $Y$.
	
	Pick a closed $r$-form $\omega \in \Omega^r(X)$ representing the class $\alpha$. Since $Y$ is exact, and since the collar neighborhood $U$ deformation retracts to $Y$, we have that $\omega|_U$ is also exact on $U$. Choose a primitive $\theta \in \Omega^{r-1}(U)$ so that $\omega|_U = d\theta$. Because $\eta$ is zero outside of $U$, we obtain a smooth $(r+k-1)$-form $\beta$ given by the extension by $0$ of the form $-\eta \wedge \theta$. One computes
	$$\eta \wedge \omega = d\beta,$$
	hence proving the third condition in the statement of the theorem: regardless of the parity of $k$ and $r$, we have $PD(Z) \wedge \alpha = 0 \in H^{r+k}(X;\R)$.
	
	Suppose now that $k$ is odd. We see that $\eta \wedge \beta = 0$ exactly. In other words, $\omega+t\beta$ is a closed element of $\scr{D}^{\infty}_{\eta}(X)$, regardless of the parity. Hence, $[\omega] \in V^{\infty}_{[\eta]} = V^{\infty}_{PD(Z)}$.
	
	Suppose now that $r$ is odd. Let $\eta_{\ell}$ denote the extension by zero of the form $\frac{1}{\ell!}\theta^{\ell} \wedge \eta$ for $\ell \geq 0$. For example, $\eta_0 = \eta$. We see that $d\eta_{\ell+1} = \omega \wedge \eta_{\ell}$ for all $\ell \geq 0$, and so $\sum_{\ell=0}^{\infty} \eta_{\ell}t^{\ell}$ is a closed element of $\scr{D}_{\omega}^{\infty}(X)$. Hence, $[\eta] \in V^{\infty}_{[\omega]}$.
\end{proof}

We record the following corollary, which follows just by considering the first obstruction ($1$-jet deformability).

\begin{cor}\label{cor:inject_is_torsion}
	Suppose we have a pair $(X,\alpha)$ with $X$ a closed oriented manifold and $\alpha \in H^r(X;\R)$ a cohomology class of homogeneous degree $r$. If the map $\Phi_{\alpha} \colon H^k(X;\R) \rightarrow H^{r+k}(X;\R)$ given by
	$$\Phi_{\alpha}(\mu) = \alpha \wedge \mu$$
	is injective, then all exact submanifolds of codimension $k$ are torsion in homology. In particular, for $k=1$, there are no non-separating exact hypersurfaces.
\end{cor}
\begin{proof}
	Since $\Phi_{\alpha}$ is assumed to be injective, this implies that $0 = PD(Z) \in H^k(X;\R)$. This implies that integrally, $PD(Z)$ is torsion in $H^k(X)$, and hence $Z$ is torsion in $H_{N-k}(X)$. The final statement follows since if $k=1$, then $H^1(X) = H_{N-1}(X)$ is non-torsion, and a hypersurface represents the trivial class $Z = 0$ if and only if it is separating.
\end{proof}

\subsection{From $L$-jet deformability to $(L+1)$-jet deformability}

One expects in general that a class being $L$-jet deformable need not imply in general that it is $(L+1)$-jet deformable. For $L=0$, this is clear: every class is $0$-jet deformable in any direction, while as not every class is $1$-jet deformable in a given nonzero direction (or perhaps any nonzero direction, e.g. in the setting of Corollary \ref{cor:inject_is_torsion}). In Example \ref{exam:nilmanifold}, we indeed saw that the class $[BT]$ was $1$-jet deformable in the direction $[A]$ but not $2$-jet deformable in the direction $[A]$.

Nonetheless, there are some situations in which one can extrapolate higher deformability. One such reason is purely a matter of degree.
\begin{prop}
	Suppose $\alpha \in H^r(X;\R)$ and $\mu \in H^k(X;\R)$ with $k>1$ odd. Let
	$$L_0 := \left\lceil \frac{N-r}{k-1} \right\rceil-1.$$
	Then the following are equivalent:
	\begin{itemize}
		\item $\alpha \in V^{\infty}_{\mu}(X)$
		\item $\alpha \in V^L_{\mu}(X)$ for every $L \geq L_0$
		\item $\alpha \in V^{L_0}_{\mu}(X)$
	\end{itemize}
\end{prop}
\begin{proof}
	The first bullet point clearly implies the second, and the second implies the third. To get from the third to the first, suppose that we have fixed representatives $[\omega] = \alpha$ and $[\eta] = \mu$, so that there is some sequence $\omega_0 = \omega, \omega_1,\ldots,\omega_{L_0}$ with $d\omega_{\ell+1} = \eta \wedge \omega_{\ell}$ for $0 \leq \ell < L_0$. Notice that for degree reasons, $|\omega_j| = r+j(k-1)$. It follows that $\eta \wedge \omega_{L_0}$ has degree $r+L_0(k-1)+k \geq N+1$, and since $X$ has dimension $N$, this is automatically $0$. Hence, we may take $\omega_j = 0$ for $j > L_0$ to yield a sequence realizing $\infty$-jet deformability.
\end{proof}

In the case $k=1$, in which case there is no a priori degree reason why $L$-jet deformability should ever imply $(L+1)$-jet deformability. Nonetheless, Example \ref{exam:top_codim} provides an interesting such example in which precisely this phenomenon happens.
\begin{prop} \label{prop:1-to-infty}
	Suppose $X^N$ is a closed oriented manifold of dimension $N$, and let $\mu \in H^1(X;\R)$ and $\alpha \in H^{N-1}(X;\R)$ be such that
	$$\alpha \in V^1_{\mu}(X)$$
	(i.e. $\mu \wedge \alpha = 0 \in H^N(X;\R)$). If $\mu$ is a rational class (i.e. in the image of the (automatically injective) extension-of-scalars map $H^1(X;\Q) \rightarrow H^1(X;\R)$), then
	$$\alpha \in V^{\infty}_{\mu}(X).$$
\end{prop}

\begin{proof}
	We have that there is some rational number $c \neq 0$ and some integral homology class $Z \in H_{N-1}(X)$ such that $\mu = c\cdot PD(Z)$.In this degree, every such $Z$ is represented by a smooth hypersurface. Let $Y$ be one such hypersurface. By the discussion of Example \ref{exam:top_codim}, the condition $\mu \wedge \alpha = c \cdot PD(Z) \wedge \alpha = 0$ is equivalent to the condition that $Y$ is exact. By Theorem \ref{thm:main}, it follows that $\alpha$ is $\infty$-jet deformable in the direction of $PD(Z)$, and the scale invariance of Proposition \ref{prop:scale_change} hence implies that $\alpha$ is $\infty$-jet deformable in the direction of $\mu$.
\end{proof}

\begin{rmk}
	We leave open in the case $k=1$ the question of whether for any fixed manifold $X$, there is a finite constant $L_X<\infty$ such that $\alpha \in V^{\infty}_{\mu}(X)$ if and only if $\alpha \in V^{L_X}_{\mu}(X) \in H^*(X;\R)$ for every class $\mu \in H^k(X;\R)$.
\end{rmk}

\color{black}

\subsection{Algebraic generalities} \label{ssec:general}

The algebraic package built in this section works quite generally. Suppose that $(\scr{A}^{\bullet},d^{\scr{A}})$ is a dg-algebra (cohomologically graded) over a field $\F$ of characteristic $0$, and that $(\scr{M}^{\bullet},d^{\scr{M}})$ is a $(\scr{A}^{\bullet},d^{\scr{A}})$-dg-module. Then for any closed element $\eta \in \scr{A}^k$, we may form the cochain complexes
$$\scr{D}^L_{\eta}(\scr{M}^{\bullet},d^{\scr{M}}) := \left(\scr{M}^{\bullet} \otimes_{\F} \quot{\F\llbracket t \rrbracket}{t^{(L+1)}},(d^{\scr{M}})^{L}_{t\eta}\right)$$
in the same way, with the understanding that $|t|=1-k$ and therefore $t^2=0$ if $k$ is even. We may as before also take $L = \infty$. These complexes are always naturally dg-modules over the dg-algebra $\scr{D}^L_{0}(\scr{A}^{\bullet},d^{\scr{A}})$. One still has truncation maps, gauge change isomorphisms, and scale change isomorphisms, all such that Proposition \ref{prop:cochain_map_properties} naturally generalizes to this setting. We therefore arrive at a notion of $L$-jet deformability for classes in $H^r(\scr{M}^{\bullet},d^{\scr{M}})$ in the direction of the class $[\eta] \in H^{k}(\scr{A}^{\bullet},d^{\scr{A}})$.

Of particular geometric interest is the case that $(\scr{A}^{\bullet},d^{\scr{A}}) = (\Omega^*(X),d)$ is the usual de Rham complex (with underlying field $\F = \R$), and $(\scr{M}^{\bullet},d^{\scr{M}}) = (\Omega^*(X;E),d_{\nabla})$ consists of differential forms of sections of a vector bundle $\pi \colon E \rightarrow X$ with a flat connection $\nabla$. For example, if $E$ is the trivial line bundle, but the holonomy of $\nabla$ is nontrivial, then one obtains questions about what might be called `locally conformal' (or perhaps more precisely, `locally homotetic') differential forms. This includes the case of locally conformal symplectic geometry; c.f. the aforementioned work of Apostolov--Dloussky \cite[Proposition 4.2]{AD} for a particular case of the deformation problem. For all such examples of cohomology classes $\alpha \in H^r(\Omega^*(X;E),d_{\nabla})$, there is an obvious notion of exact submanifolds, and one obtains an analogue of Theorem \ref{thm:main} which we omit in this article.

\section{Symplectic applications} \label{sec:symp_app}

Recall that a symplectic manifold $(X,\omega)$ is a smooth manifold $X$ with a closed $2$-form $\omega \in \scr{Z}^2(X)$ which is non-degenerate, meaning that the interior product $\iota_{\bullet}\omega \colon TX \rightarrow T^*X$ given by $v \mapsto \iota_{v}\omega$ is a bundle isomorphism. Equivalently, $X$ has even dimension $N=2n$ and $\omega^n$ is nowhere vanishing.

One particularly useful type of hypersurface in symplectic geometry (and especially neck-stretching arguments in symplectic field theory, c.f. \cite{EGH,BEHWZ}) is known as a contact-type hypersurface.
\begin{defn}
	A hypersurface of a symplectic manifold $H^{2n-1} \subset (X^{2n},\omega)$ is said to be of \textbf{contact-type} if there exists some $\alpha \in \Omega^1(H)$ such that the following two properties hold:
	\begin{itemize}
		\item $d\alpha = \omega|_H$
		\item $\alpha \wedge d\alpha^{n-1}$ is nowhere vanishing
	\end{itemize}
\end{defn}
By the first property, it is clear that any contact-type hypersurface is automatically exact. In the rest of this section, we present three classes of symplectic manifolds admitting no non-separating exact hypersurfaces, and hence also no non-separating contact-type hypersurfaces: K\"ahler manifolds, symplectically uniruled manifolds, and the Kodaira--Thurston manifold.

\begin{rmk}
	It is known that there are closed symplectic manifolds admitting non-separating contact-type hypersurfaces, though these are typically constructed using h-principle techniques, which often leads to existence results without concrete examples. The main construction is due to Etnyre and outlined in work of Albers--Bramham--Wendl \cite[Example 1.3]{ABW} and Wendl \cite[Section 1.3.1]{Wendl}.
\end{rmk}

\begin{rmk}
	Other results related to obstructions for non-separating contact-type hypersurfaces appear in work of Albers--Bramham--Wendl \cite{ABW} as well as Mukherjee \cite{Mu}. We leave open whether their geometric restrictions on the types of hypersurfaces which appear in their hypotheses imply non-vanishing obstructions of the type in this article, which would yield an alternative proof of their results.
\end{rmk}

\subsection{K\"ahler manifolds}

Recall that a K\"ahler manifold $(X,\omega,J)$ is a manifold $X$ together with a symplectic form $\omega$ and an integrable complex structure $J$ which are compatible, meaning:
\begin{itemize}
	\item $\omega(v,Jv) > 0$ for all $v \neq 0$
	\item $\omega(Jv,Jw) = \omega(v,w)$
\end{itemize}
We call a cohomology class $\alpha \in H^2(X;\R)$ a K\"ahler class if it can be realized as the cohomology class $[\omega]$ for a K\"ahler structure $(\omega,J)$. Note that because K\"ahler manifolds are symplectic/complex, they are automatically even-dimensional, so it is typical to write $N = 2n$ for the dimension.

K\"ahler manifolds satisfy the Hard Lefschetz theorem, which states that for any $k \leq n$, wedge product with $[\omega]^{n-k}$ yields an isomorphism
$$H^{k}(X;\R) \cong H^{2n-k}(X;\R).$$
In particular, for any $\ell \leq n-k$, wedge product with $[\omega]^{\ell}$ is injective as a map $H^k(X;\R) \rightarrow H^{k+2\ell}(X;\R)$. We obtain the following theorem.

\begin{thm} \label{thm:Kaehler}
	Suppose $X^{2n}$ is a manifold with $\alpha \in H^2(X;\R)$ a K\"ahler class. For $k \leq n$ and $\ell \leq n-k$, any submanifold of codimension $k$ which is exact with respect to $\alpha^{\ell}$ is automatically torsion in homology $H_{2n-k}(X)$. In particular, taking $k=1$, there is no non-separating hypersurface which is exact with respect to $\alpha^{\ell}$ for $\ell \leq n-1$.
\end{thm}
\begin{proof}
	By the Hard Lefschetz theorem, wedge product with $\alpha^{\ell}$ is injective as a map $H^j(X;\R)\rightarrow H^{j+2\ell}(X;\R)$, and hence Corollary \ref{cor:inject_is_torsion} applies.
\end{proof}

\begin{cor} \label{cor:Kaehler}
	K\"ahler manifolds of dimension $\geq 4$ admit no non-separating exact hypersurfaces, and in particular no non-separating contact-type hypersurfaces.
\end{cor}
\begin{proof}
	Apply Theorem \ref{thm:Kaehler} with $k=\ell = 1$.
\end{proof}

\subsection{Symplectically uniruled manifolds} \label{ssec:symp_uni}

Recall that on a symplectic manifold, one may define \textbf{Gromov--Witten invariants} of the form
$$\mathrm{Gr}_{g,m,A}^{(X,\omega)} \in H_d(\overline{\scr{M}}_{g,m} \times X^m;\Q),$$
where $g \in \Z_{\geq 0}$, $m \in \Z_{\geq 0}$, $A \in H_2(X)$, $\overline{\scr{M}_{g,m}}$ the Deligne--Mumford moduli space of stable nodal Riemann surfaces of genus $g$ with $m$ marked points, and the degree $d$ on the right-hand side is determined by the values $g$, $m$, and $A$. Intuitively, the Gromov-Witten invariant is defined by studying the (compactified) moduli space of genus $g$ pseudo-holomorphic maps with $m$ marked points representing the homology class $A$ (for a generic choice of compatible almost complex structure $J$ on $(X,\omega)$). The expected dimension of the moduli space of pseudo-holomorphic maps is $d$ (called the virtual dimension), and the moduli space comes with an obvious evaluation map to $\overline{\scr{M}}_{g,m} \times X^m$, so the push-forward of the fundamental class of the moduli space is the desired Gromov--Witten invariant. 

Given homology classes $\alpha_1,\ldots,\alpha_m \in H_*(X;\Q)$ and $\beta \in H_*(\overline{\scr{M}}_{g,m})$ of total codimension $d$, one obtains by intersection product an element in $H_0(\overline{\scr{M}}_{g,m} \times X^m;\Q) \cong \Q$, for which we write the value, also called a \textbf{Gromov--Witten invariant},
$$\mathrm{Gr}_{g,m,A}^{(X,\omega)}(\alpha_1,\ldots,\alpha_m;\beta) \in \Q,$$
Intuitively, these count the number of pseudo-holomorphic maps from a genus $g$ Riemann surface with $m$ marked point such that those $m$ marked points evaluate to points on generic representatives of the $\alpha_j$ and such that the conformal class of the fundamental class of the corresponding moduli space pushes forward to $\beta$. (It is typically easiest to think about the case where $\beta$ is itself the fundamental class of $H^*(\overline{\scr{M}}_{g,m})$.)

\begin{rmk}
One must provide analytical details in order to make this precise, for which there are many strategies \cite{CM,FO,HWZ,LT,Pardon,Ruan}. In an effort to keep this paper shorter (and because the focus of the paper is not on these technical details), we shall assume (without further statement) that $[\omega] \in H^2(X;\R)$ is an integral cohomology class, so that we may use the strategy of Cieliebak and Mohnke \cite{CM}. We will only use it in the proof of Lemma \ref{lem:symp_uniruled_H3}, which is outsourced to work of Wendl \cite{Wendl} who uses the Cieliebak--Mohnke machinery.
\end{rmk}

\begin{defn}
	A symplectic manifold $(X,\omega)$ is said to be \textbf{symplectically uniruled} if there exists some $m \geq 3$, $A \in H_2(X)$, $\alpha_2,\ldots,\alpha_m \in H_*(X;\Q)$, and $\beta \in H_*(\overline{\scr{M}}_{0,m};\Q)$, such that
	$$\mathrm{Gr}_{g,m,A}^{(X,\omega)}([pt],\alpha_2,\ldots,\alpha_m;\beta) \neq 0 \in \Q$$
	where $[pt] \in H_0(X;\Q)$ is the homology class of a point.
\end{defn}

Wendl \cite{Wendl} proved that such manifolds do not admit any non-separating contact-type hypersurfaces. In fact, he proves a slightly more general result, for which we introduce a definition.

\begin{defn}
	A hypersurface of a symplectic manifold $H^{2n-1} \subset (X^{2n},\omega)$ is said to be \textbf{pseudoconvex} if there exists a compatible almost complex structure $J$ such that $\xi := TH \cap JTH \leq TH$ induces a contact structure on $H$, i.e. for any nonzero $1$-form $\alpha \in \Omega^1(H)$ with $\ker \alpha = \xi$, we have $\alpha \wedge d\alpha^{n-1}$ is a volume form.
\end{defn}

\begin{thm}[Wendl \cite{Wendl}] \label{thm:Wendl_pseudoconvex}
	If $(X,\omega)$ is a closed symplectic manifold that is symplectically uniruled, then there exists no non-separating pseudoconvex hypersurface.
\end{thm}

We may complement this by extending Wendl's theorem to exact hypersurfaces. The key lemma follows essentially from Wendl's work.

\begin{lem} \label{lem:symp_uniruled_H3}
	If $(X^{2n},\omega)$ is symplectically uniruled, then $\Phi_{\omega} \colon H^1(X;\R) \rightarrow H^3(X;\R)$, given by wedge product with $[\omega]$, is injective.
\end{lem}
\begin{proof}
	Suppose $Z \neq 0 \in H_{2n-1}(X)$, so that $PD(Z) \neq 0 \in H^1(X)$. In his paper, Wendl \cite[Lemma 3.2]{Wendl} proves that one may find a map $\Phi \colon S^1 \times S^2 \rightarrow (X,\omega)$ with the properties that:
	\begin{itemize}
		\item $\Phi_*[S^1 \times \{*\}] \bullet Z \neq 0 \in H_0(X;\Q)$
		\item The restriction to each fiber of $\Phi$ is a pseudo-holomorphic (nodal) sphere $S^2 \rightarrow (X,\omega,J)$ with respect to some compatible almost complex structure $J$ in the homology class $A$ (i.e. the homology class appearing in the nonzero Gromov--Witten invariant).
	\end{itemize}
	The first point may be equivalently formulated as the fact that $\Phi^*(PD(Z))|_{S^1 \times \{*\}} \neq 0$. The second gives that $\Phi^*([\omega])|_{\{*\} \times S^2} \neq 0$ (and is the class of area $\langle \omega, A \rangle$). By the K\"unneth theorem, we have that $H^3(S^1 \times S^2;\R) = H^1(S^1;\R) \otimes H^2(S^2;\R)$. In particular, the element $\Phi^*(PD(Z) \wedge [\omega]) \in H^3(S^1 \times S^2;\R)$ is identified with the nonzero element $\Phi^*(PD(Z))|_{S^1} \otimes \Phi^*([\omega])|_{S^2} \neq 0$. Hence,
	$$\langle PD(Z) \wedge [\omega], \Phi_*([S^1 \times S^2]) \rangle \neq 0,$$
	and so $PD(Z) \wedge [\omega] \neq 0$. It follows that $\Phi_{\omega}(PD(Z)) \neq 0$ for integral classes $Z$. Tensoring with $\R$, we obtain the injectivity of $\Phi_{\omega}$.
\end{proof}

\begin{thm} \label{thm:symp_uniruled}
	If $(X,\omega)$ is a closed symplectic manifold that is symplectically uniruled, then there exists no non-separating exact hypersurface.
\end{thm}

\begin{proof}
	Combine Lemma \ref{lem:symp_uniruled_H3} and Corollary \ref{cor:inject_is_torsion}.
\end{proof}

\subsection{The Kodaira--Thurston manifold}

The Kodaira--Thurston manifold is a symplectic $4$-dimensional manifold, initially constructed by Thurston \cite{Thurston} for the purpose of producing a symplectic manifold with no K\"ahler form. Smoothly, it is just the manifold of Example \ref{exam:nilmanifold}. The corresponding symplectic form is just $\omega = AC+BT$ (as a left-invariant form), which is easily checked to be non-degenerate and closed, and hence itself a symplectic form. We are now able to prove our desired result.

\begin{thm}\label{thm:KT} The Kodaira--Thurston symplectic manifold $(X,AC+BT)$ admits no non-separating exact hypersurface.
\end{thm}

\begin{proof}
	Suppose that there is an exact hypersurface Poincar\'e dual to the cohomology class $[\eta]$ of the form $\eta = pA+qB+rT$. We have that
	$$\eta \wedge \omega = pABT-qABC-rBCT$$
	must be exact by the first obstruction. This implies that $q=r=0$, and hence $\eta = pA$. In Example \ref{exam:nilmanifold}, we proved that
	$$V^{\infty}_{A} = V^{2}_A = \mathrm{span}([AC],[AT]) \leq H^2(X;\R).$$
	We see that $[AC+BT] \notin V^{\infty}_{A}$. By Proposition \ref{prop:scale_change}, we in fact have that
	$$[AC+BT] \notin V^{\infty}_{pA}$$
	for any $p \neq 0$. Hence, we have that any exact hypersurface must be in the trivial homology class, and is therefore separating.
\end{proof}

\bibliography{Bib}{}
\bibliographystyle{plain}

\end{document}